%% file: unit-fractions-calc-time.tex
\documentclass[a4paper,reqno]{amsart}
\pdfoutput=1

\usepackage[utf8]{inputenc}
\usepackage[T1]{fontenc}
\usepackage{lmodern}
\usepackage[draft=false,kerning=true]{microtype}

\usepackage[english]{babel}

\usepackage{algorithm}
\usepackage{colonequals}
\usepackage{mathtools}
\usepackage{todonotes}

\usepackage{hyperref}
\usepackage{fancyvrb}
\DeclarePairedDelimiter{\abs}{\lvert}{\rvert}
\newcommand{\bfA}{\mathbf{A}}
\newcommand{\bfD}{\mathbf{D}}
\newcommand{\bfM}{\mathbf{M}}
\newcommand{\bfO}{\mathbf{O}}
\newcommand{\bfS}{\mathbf{S}}

\newcommand{\eps}{\varepsilon}
\DeclarePairedDelimiterXPP{\f}[2]{\foperator{#1}}(){}{#2}
\DeclarePairedDelimiterXPP{\fexp}[1]{\exp}(){}{#1}
\DeclarePairedDelimiter{\floor}{\lfloor}{\rfloor}
\newcommand{\foperator}[1]{\mathop{{#1}\empty{}}}
\DeclarePairedDelimiter{\hp}{[}{]}
\renewcommand{\MR}[1]{}
\newcommand{\N}{\mathbb{N}}
\newcommand{\OEIS}[1]{\href{http://oeis.org/#1}{#1}}
\DeclarePairedDelimiterXPP{\Oh}[1]{\foperator{O}}(){}{#1}

\DeclarePairedDelimiterX{\setm}[2]{\{}{\}}{#1 \colon \mathopen{}#2}
\newcommand{\Z}{\mathbb{Z}}

\newtheorem{theorem}{Theorem}

\newtheorem{lemma}{Lemma}[section]

\theoremstyle{remark}

\newtheorem{remark}[lemma]{Remark}

\numberwithin{equation}{section}
\numberwithin{figure}{section}
\numberwithin{table}{section}

\input{pygments}
\begin{document}

\title{Algorithmic counting of nonequivalent compact Huffman codes}

\author{Christian Elsholtz}

\address{\parbox{12cm}{%
    Christian Elsholtz \\
    Institute of Analysis and Number Theory \\
    Graz University of Technology \\
    Kopernikusgasse 24, A-8010 Graz, Austria \\}} 

\email{\href{mailto:elsholtz@math.tugraz.at}{elsholtz@math.tugraz.at}}

\author{Clemens Heuberger}

\address{\parbox{12cm}{%
    Clemens Heuberger \\
    Department of Mathematics \\
    Alpen-Adria-Universit\"at Klagenfurt \\
    Universit\"atsstra\ss e 65--67, A-9020 Klagenfurt am W\"orthersee, Austria \\}}

\email{\href{mailto:clemens.heuberger@aau.at}{clemens.heuberger@aau.at}}

\author{Daniel Krenn}

\address{\parbox{12cm}{%
    Daniel Krenn\\
    Department of Mathematics \\
    Paris Lodron University of Salzburg \\
    Hellbrunnerstra\ss e 34, A-5020 Salzburg, Austria \\}}

\email{\href{mailto:math@danielkrenn.at}{math@danielkrenn.at} \textit{or}
  \href{mailto:daniel.krenn@plus.ac.at}{daniel.krenn@plus.ac.at}}

\thanks{C.~Elsholtz is supported by the Austrian Science Fund (FWF): W1230
  and by Project Arithrand of the Austrian Science Fund (FWF): I 4945-N and of ANR-20-CE91-0006.
  C.~Heuberger and D.~Krenn are supported
  by the Austrian Science Fund (FWF): P28466-N35.}

\keywords{unit fractions, Huffman codes, $t$-ary trees, counting, generating
  function}

\subjclass[2020]{05A15; 05C05, 05C30, 11D68, 68P30}

\begin{abstract}
  It is known that the following five counting problems lead to the
  same integer sequence~$\f{f_t}{n}$:
  \begin{enumerate}
  \item the number of nonequivalent compact Huffman codes of
    length~$n$ over an alphabet of $t$ letters,
  \item  the number of ``non\-equi\-va\-lent'' complete rooted $t$-ary
    trees (level-greedy trees) with $n$~leaves,
  \item the number of ``proper'' words (in the sense of Even and Lempel),
  \item the number of bounded degree
    sequences (in the sense of Koml{\'o}s, Moser, and Nemetz), and
  \item the number of ways of writing
  \begin{equation*}
    1= \frac{1}{t^{x_1}}+ \dots + \frac{1}{t^{x_n}}
  \end{equation*}
  with integers $0 \leq x_1 \leq x_2 \leq \dots \leq x_n$.
  \end{enumerate}
  In this work, we show that one can compute
  this sequence for \textbf{all} $n<N$ with essentially one power series
  division. In total we need at most $N^{1+\eps}$ additions and
  multiplications of integers of $cN$ bits (for a positive constant $c<1$ depending on $t$ only) or
  $N^{2+\eps}$ bit operations, respectively, for any~$\eps>0$.
  This improves
  an earlier bound by Even and Lempel who needed $\Oh{N^3}$ operations
  in the integer ring or $\Oh{N^4}$ bit operations, respectively.
\end{abstract}

\maketitle

\section{Introduction}
\label{sec:introduction}

\subsection*{Motivation}

The purpose of this paper is to study the complexity of a counting
problem, namely determining the number of nonequivalent compact
Huffman codes of length~$n$ over an alphabet of $t$~letters, and
several equivalent combinatorial or number theoretic objects; see
below and in particular~\eqref{eq:f_t} for a precise definition.
The fastest algorithm in the published literature is due to
Even and Lempel~\cite{Even-Lempel:1972:gener} (1972) and has a
complexity of $\Oh{N^3}$ operations in the ring of integers.

When actually computing the number of such compact Huffman codes, we
experimentally observed that an approach of evaluating a generating
function---this generating function was first studied by Flajolet and
Prodinger~\cite{Flajolet-Prodinger:1987:level} (1987)---appears to be
very fast. A detailed analysis (see Theorem~\ref{thm:extract-coeff})
shows that the complexity
is indeed only $\Oh{N^{1+\eps}}$ (for any~$\eps>0$)
additions and multiplications
of integers of size $cN$ bits (see~\eqref{eq:asy-coefficients}),
where $c<1$ is a positive constant
depending on $t$ only.

In this paper, we will first describe the different but equivalent
objects that we count and then present a quite detailed analysis of
computing the number of these objects.

\subsection*{Codes, unit fractions and more}
For a fixed integer $t\ge 2$, Elsholtz, Heuberger and Prodinger
\cite{Elsholtz-Heuberger-Prodinger:2013:huffm} studied the number
\begin{equation}
  \label{eq:f_t}
  f_t(r) \colonequals \abs[\bigg]{\setm[\bigg]{(x_1, \ldots, x_r)\in \Z^r}{
      0\le x_1\le \cdots \le x_r,\, \sum_{i=1}^r \frac{1}{t^{x_i}}=1}},
\end{equation}
$r\ge 0$, i.\,e.\ the number of partitions of $1$ into nonpositive powers of $t$.
It is known that this counting problem is equivalent to several other
counting problems, namely the number of ``non\-equi\-va\-lent'' complete rooted
$t$-ary trees (also called ``level-greedy trees'';
see~\cite{Elsholtz-Heuberger-Prodinger:2013:huffm, Flajolet-Prodinger:1987:level, Heuberger-Krenn-Wagner:2013:analy-param}),
the number of ``proper words'' (in the sense of Even and
Lempel~\cite{Even-Lempel:1972:gener}), the number of bounded degree sequences
(in the sense of Koml{\'o}s, Moser, and Nemetz~\cite{Komlos-Moser-Nemetz:1984}),
and the number of nonequivalent compact Huffman codes%
\footnote{A \emph{Huffman code} over an alphabet of $t$ letters is a
  prefix-free subset (the set of ``code words'') of the set of finite words
  over this alphabet, i.e., no code word is a prefix of another code
  word. It is said to be \emph{compact} if no further code word can be
  added without violating the prefix-freeness condition.
  Two compact Huffman codes are considered to be \emph{equivalent} if
  the multisets of the lengths of the code words are equal, and one
  can choose a representative where shorter words are
  lexicographically smaller than longer words.
}
of length~$r$ over an
alphabet of $t$ letters. For a detailed survey on the existing results,
applications and literature on these sequences; see
\cite{Elsholtz-Heuberger-Prodinger:2013:huffm}. As a small concrete example, we
note that for $t=2$, $r=5$ we have $f_2(5)=3$, as can be seen from working out
the following:
\begin{equation*}
  1 = \frac{1}{2}+ \frac{1}{4}+\frac{1}{8}+\frac{1}{16}+\frac{1}{16}
  = \frac{1}{2}+ \frac{1}{8}+\frac{1}{8}+\frac{1}{8}+\frac{1}{8}
  = \frac{1}{4}+ \frac{1}{4}+\frac{1}{4}+\frac{1}{8}+\frac{1}{8}.
\end{equation*}

As discussed in~\cite{Elsholtz-Heuberger-Prodinger:2013:huffm},
$\f{f_t}{r}$ is positive only when $r=1+n(t-1)$, so it is more convenient to
study $\f{g_t}{n}=f_t(1+n(t-1))$ instead. For $t=2$ the values of $\f{g_t}{n}$
start with
\begin{equation*}
  1,1,1,2,3,5,9,16,28,50,89,159,\ldots,
\end{equation*}
for $t=3$ with
\begin{equation*}
  1,1,1,2,4,7,13,25,48,92,176,338,\ldots,
\end{equation*}
and the first terms of $\f{g_4}{n}$ are
\begin{equation*}
  1,1,1,2,4,8,15,29,57,112,220,432,\ldots.
\end{equation*}
These are sequences \OEIS{A002572}, \OEIS{A176485} and \OEIS{A176503} in the
On-Line Encyclopedia of Integer Sequences~\cite{OEIS:2013}.

\subsection*{Asymptotics}
It has been proved (see Elsholtz, Heuberger, Prodinger
\cite{Elsholtz-Heuberger-Prodinger:2013:huffm}) that for fixed $t$, the
asymptotic growth of these sequences can be described by two main terms and an
error term as
\begin{equation}\label{eq:gt-asymptotics}
  \f{g_t}{n} = R \rho^n + R_2 \rho_2^n + \Oh[\big]{r_3^n},
\end{equation}
where $1 < r_3 < \rho_2 < \rho < 2$. Here all constants depend on~$t$. In
particular, if $t=2$, then
\begin{align*}
  \rho&=1.794\ldots, & \rho_2&=1.279\ldots, & 
  r_3&=1.123, & R&=0.14\ldots, & R_2&=0.061\ldots.
\end{align*}
Moreover, the authors of~\cite{Elsholtz-Heuberger-Prodinger:2013:huffm}
also show that $\rho = 2-2^{-t-1}+\Oh[\big]{t\,4^{-t}}$ as $t\to\infty$.

Beside the enumeration of all these objects,
probabilistic questions concerning many different parameters have been studied
asymptotically
in~\cite{Heuberger-Krenn-Wagner:2013:analy-param,
  Heuberger-Krenn-Wagner:ta:analy-param-full}.

\subsection*{Algorithmic counting}

As this family of sequences appears in many different
contexts and as the sequences' growth rates have been studied in
detail (see the section above and the introduction
of~\cite{Elsholtz-Heuberger-Prodinger:2013:huffm} for full details),
it is somehow surprising that the current record on the algorithmic
complexity of determining the members of the sequence (in the case
$t=2$) appears to be a 50 years old paper by Even and
Lempel~\cite{Even-Lempel:1972:gener}. Hence it seemed worthwhile to
study this complexity from a new point of view and we thus succeeded
to improve the upper bound complexity considerably; see section below.

The algorithm of Even and Lempel~\cite{Even-Lempel:1972:gener} produces
the sequence $\f{g_t}{n}$ for $n<N$. It takes $\Oh{N^3}$ additions
of integers bounded by $\Oh{\rho^N}$ (with $\rho<2$; so integers with
roughly $N$ bits in size), which
are $\Oh{N^4}$ bit operations.
They only studied the case $t=2$ in detail, but mention
that their result can be generalized to arbitrary~$t$.

\subsection*{Main result}

In this paper we take an entirely new approach to the problem of
evaluating~$g_t(n)$. Rather than thinking about an algorithm itself,
as Even and Lempel~\cite{Even-Lempel:1972:gener} did, we think about
how to evaluate the
generating function~\eqref{eq:power-series} of~$g_t(n)$ established in~\cite{Elsholtz-Heuberger-Prodinger:2013:huffm}
efficiently. As it
turns out the cost essentially comes from one division of power series of
precision\footnote{We say that a power series~$H$ has precision~$N$
  if we can write it as
  $H(q) = \sum_{n=0}^{N-1} h_n q^n + \Oh{q^N}$
  with explicit coefficients~$h_n$.}~$N$
whose coefficients are integers bounded by $\Oh{\rho^N}$ (with $\rho<2$).

Estimating the cost of this evaluation strategy leads to tremendous
improvement---to be precise, by a factor~$N^2$ in both ring operations
and bit operations---of the cost of using~\cite{Even-Lempel:1972:gener}.
It is not obvious that the cost for evaluation of numerator and
denominator of the generating function are asymptotically (much)
smaller than the total cost; see Theorem~\ref{thm:extract-coeff} for
details and also Section~\ref{sec:proofs} providing even more details during the proof of this theorem.
We in particular show that the cost for evaluating numerator and
denominator are asymptotically almost (neglecting logarithmic factors)
by a factor~$N$ smaller.

Using the
multiplication algorithms of Schönhage and
Strassen~\cite{Schoenhage-Strassen:1971:schnelle-mult}, of
Fürer~\cite{Fuerer:2007:faster-integer-mult, Fuerer:2009:faster-integer-mult},
or of Harvey and van der Hoeven~\cite{Harvey-VanDerHoeven:2019:integer-mult-nlogn}
(see Section~\ref{sec:cost-add-mult} for an overview) our algorithm leads to
$N (\log N)^2 \, 2^{\Oh{\log^* N}}$ operations in the integer ring and consequently
$N^2 (\log N)^4 \, 2^{\Oh{\log^* N}}$ bit operations, where $\log^* N$
denotes the \emph{iterated logarithm}.%
\footnote{The \emph{iterated logarithm}
  (also called \emph{log star}) gives
  the number of applications of the logarithm so that the result is at
  most~$1$. For example, we can define it recursively by $\log^* M =
  1+\log^*(\log M)$ if $M>1$ and $\log^* M = 0$ otherwise.}
In Remark~\ref{rem:space} a
discussion on the memory requirements can be found. An implementation of this
algorithm, based on FLINT~\cite{flint:2015:2.5.2, Hart:2010:FLINT-intro}
(which is, for example, included in the SageMath mathematics
software~\cite{SageMath:2018:8.3}) is also available;\footnote{The
  code accompanying this article can be downloaded from
  \url{https://gitlab.com/dakrenn/count-nonequivalent-compact-huffman-codes}.}
see also Appendix~\ref{sec:code} for the relevant lines of code and
remarks related to the implementation. In Appendix~\ref{sec:timing},
we discuss the running times of this implementation.

The literature describes a number of algorithms constructing the
complete list of $t$-ary Huffman codes of length $r=1+n(t-1)$; see
\cite{Hoffman-Johnson-Wilson:2005:gener-huffm,
  Khosravifard-Esmaeili-Saidi-Gulliver:2003:tree-compact-codes,
  Narimani-Khosravifard:2008:supertree-compact-codes,
  Niyaoui-Reda:2016:alg-gen-binary-huffman-codes}.
There is no performance analysis given. But, as the number of such
codes grows exponentially in~$r$ it is clear that listing all codes is
not a fast method to determine the number of such codes only. The
algorithm by Even and Lempel~\cite{Even-Lempel:1972:gener} computes
the number~$f_2(n)$ without listing all codes, and is to the best of
our knowledge the fastest algorithm previously known.
Our algorithm relies on calculations involving power series with large integer
coefficients.

It should also be emphasized that the output~$f_t(n)$ of the algorithm
grows exponentially in~$n$ (this was mentioned above), therefore the
number of bits to represent~$f_t(n)$ is linear in~$n$ whereas
the input is only logarithmic in~$n$.
The quite general survey paper by
Klazar~\cite{Klazar:2018:what-is-answer} studies classes of problems
where the output needs at most a polynomial number of steps, in terms
of the combined size of input and output. As we can compute $f_t(n)$ efficiently, this
problem falls into the class considered by Klazar.

\subsection*{Notes}
It should be pointed out that in this article, we derive and compare
upper bounds. It might be that the actual cost are smaller. However,
as we compute the first $N$~coefficients all at the same time and the
coefficients grow exponentially in~$N$, a lower bound for
the number of bit operations necessarily contains a
factor~$N^2$. Moreover, as multiplication of some sort is involved,
lower order factors (growing with~$N$) are expected as well.

We also mention that the following is open: How fast can a single
coefficient~$\f{g_t}{n}$ (in contrast to all coefficients with $n<N$)
be computed?

\section{Cost of the underlying operations}
\label{sec:cost-add-mult}

In this section, we give a brief overview on the time requirements for
performing addition and multiplication of two integers and for performing
multiplication and division of power series. The current state of the art
is also summarized in Table~\ref{tab:cost-ops}.

\begin{table}
  \centering
  \begin{tabular}{lll}
    Task & Ring operations & Bit operations \\
    \hline
    addition & $N$ & $NM$ \\
    multiplication & $N\log N\, 2^{\Oh{\log^* N}}$
                   & $N\log N\, 2^{\Oh{\log^* N}} \cdot M\log M$ \\
    division & $N\log N\, 2^{\Oh{\log^* N}}$
             & $N^2 M (\log N)^2 2^{\Oh{\log^* N}}$ \\
  \end{tabular}
  \vspace*{1em}
  \caption{Cost of operations of power series with precision~$N$
    and coefficients with bit size~$M$. (We assume $M=\Oh{N}$
    and state a simpler expression for the bit operations for division.)}
  \label{tab:cost-ops}
\end{table}

\subsection*{Addition and multiplication}
First, assume that we want to perform addition of two numbers bounded by $2^M$,
i.e., numbers with $M$ bits. We have to look at each bit of the
numbers exactly once and add those (maybe with a carry). Therefore, we need
$\Oh{M}$ bit operations.

Next, we look at multiplication of two numbers bounded by $2^M$. It is clear that
this can be achieved with $\Oh{M^2}$ operations, but it can be done better. An
overview is given in the survey article by
Karatsuba~\cite{Karatsuba:1995:complexity-comp:survey}. The Karatsuba
multiplication algorithm~\cite{Karatsuba-Ofman:1962:multipl,
  Karatsuba-Ofman:1963:multipl:translation} has a complexity of~$\Oh{M^{\log_2
    3}}$. A faster generalisation of it is the
Toom--Cook-algorithm~\cite{Cook:1966:phd:min-comp-time}. Combining Karatsuba
multiplication with the Fast Fourier Transform algorithm (see Cooley and
Tukey~\cite{Cooley-Tukey:1965}) gives an algorithm with bit complexity
$\Oh{M(\log M)^{(2+\eps)}}$; see~\cite{Borodin-Munro:1975:comp-complexity,
  Borwein-Borwein-Bailey:1989, Brigham:1974:fft, Knuth:1998:Art:2}.

The multiplication algorithm given by Schönhage and Strassen
(see~\cite{Schoenhage-Strassen:1971:schnelle-mult}) takes $\Oh{M\log M \log\log
  M}$ time. It also uses fast Fourier transform. An asymptotically even faster
multiplication algorithm is given by
Fürer~\cite{Fuerer:2007:faster-integer-mult, Fuerer:2009:faster-integer-mult}.
It has computational complexity $M\log M\, 2^{\Oh{\log^* M}}$, where we again denote
the \emph{iterated logarithm} by $\log^* M$. Fürer's algorithm
uses complex arithmetic. A related algorithm of the same complexity but using
modular arithmetic is due to De, Kurur, Saha and
Saptharishi~\cite{De-Kurur-Saha-Saptharishi:2008:fast-mult, De-Kurur-Saha-Saptharishi:2013:fast-mult}.

The asymptotically fastest known multiplication algorithm is due to
Harvey and van der Hoeven~\cite{Harvey-VanDerHoeven:2019:integer-mult-nlogn};
it has a computational complexity of $\Oh{M\log M}$.

\subsection*{Power series operations}
Let us also summarize the complexity of power series computations; for
references see the books of Cormen, Leiserson, Rivest and
Stein~\cite{Cormen-Leiserson-Rivest-Stein:2001} or
Knuth~\cite{Knuth:1998:Art:2}. The multiplication can, again, be speeded up by
using fast Fourier transform. We can use the algorithms for integer
multiplication presented above; see von zur Gathen and
Gerhard~\cite{Gathen-Gerhard:2003:ModernComputerAlgebra}. Also, the
computational complexity can be improved: Given power series with precision~$N$
(i.e., the first~$N$ terms) over a ring, we can perform
multiplication with $N\log N\, 2^{\Oh{\log^* N}}$ ring operations using
Fürer's algorithm.

In order to perform division (inversion) of power series with precision~$N$, we
can use the Newton--Raphson-method. We need at most $4\f{m}{N} + N$ ring
operations, where $\f{m}{N}$ denotes the number of operations needed to
multiply two power series with precision~$N$; see von zur Gathen and
Gerhard~\cite[Theorem~9.4]{Gathen-Gerhard:2003:ModernComputerAlgebra} for
details; the additional summand $\f{m}{N}$ in comparison to that theorem comes
from the multiplication with the numerator. Therefore, by using Fürer's algorithm, we can invert/divide
with $N\log N\, 2^{\Oh{\log^* N}}$ ring operations.

The bit size occuring in the ring operations for a division of power series
with precision~$N$ and coefficients of bit size $M$ is $NM$ by the remarks
after \cite[Theorem~9.6]{Gathen-Gerhard:2003:ModernComputerAlgebra}.
Therefore and by assuming~$M=\Oh{N}$ for simpler expressions with
respect to the logarithms, we end up with
\begin{equation*}
  N \log N \, 2^{\Oh{\log^* N}} \cdot NM \log\bigl(NM\bigr)
  = N^2 M (\log N)^2 2^{\Oh{\log^* N}}
\end{equation*}
bit operations.

\section{Cost for extracting coefficients}
\label{sec:thm}

Our main result gives the number of operations needed for extracting the
coefficients~$\f{g_t}{n}$ for all $n<N$. It reflects three
different aspects: First, we count operations on a high
level, for example power series multiplications. (Below we will
denote this operation by~$\bfM$.) Second, we count operations in
the ring of integers. There, to stick with the example
on power series multiplication, the precision of the
power series is taken into account, but not the actual size of the integer.
Finally and third, we count bit
operations, where also the size of the coefficients (which are integers)
is taken into account.

Let us make this more precise and start with the high level
operations. We denote
\begin{itemize}
\item an \emph{addition} (or a \emph{subtraction}) of two power
  series by~$\bfA$,
\item a \emph{multiplication} of two power series
  by~$\bfM$, and
\item a \emph{division} of two power series
  by~$\bfD$.
\end{itemize}
As we compute the first~$N$ terms, we may assume that all power series
are of precision~$N$.

An overview and summary of the number of ring operations and bit
operations of these high level operations is provided in
Section~\ref{sec:cost-add-mult}. Clearly, we
have to deal with the size of the coefficients.
We first note that for $n<N$ each coefficient $\f{g_t}{n}$ can be
written with $M \colonequals \floor{\log_2 \f{g_t}{N}}+1$ bits and that by
using the asymptotics~\eqref{eq:gt-asymptotics} we can bound this
by
\begin{equation}\label{eq:asy-coefficients}
  N\log_2\rho+\Oh{1}
\end{equation}
when $N$ tends to $\infty$. Here the constant
$\rho<2$ depends on $t$; see
\cite{Elsholtz-Heuberger-Prodinger:2013:huffm} for details on $\rho$.

Summarizing, all the operations~$\bfA$, $\bfM$ and~$\bfD$
are performed on power series of precision~$N$ with
coefficients written by $M$ bits (numbers bounded by
$2^M$), and the cost (number of bit operations) are stated in
Section~\ref{sec:cost-add-mult}. There is one important remark at this point,
namely, we will see during our main proof
(Section~\ref{sec:proofs}) that the coefficients appearing in power
series additions and multiplications are actually much smaller than
coefficients written by~$M$ bits; we will
take this into account for counting bit operations.

Beside these main power series operations,
we additionally denote
\begin{itemize}
\item \emph{other power series operations} of precision~$N$ (for example,
  memory allocation or writing initial values) by~$\bfS$, and
\item \emph{other operations}, more precisely operations of numbers
  with less than $\log_2 N$ bits (for example additions of indices)
  by~$\bfO$.
\end{itemize}
Thus, an operation~$\bfO$ is performed on numbers bounded by~$N$ only (in
contrast to the bounded-by-$2^M$-operations).

With these notions and by collecting operations as formal sums of
$\bfA$, $\bfM$, $\bfD$, $\bfS$ and $\bfO$,
we can write down the precise formulation of our main
theorem.

\begin{theorem}\label{thm:extract-coeff}
  Calculating the first $N$ terms of $\f{g_t}{n}$ can be done with
  \begin{equation}\label{eq:cost-ps-operations}
    \bfD
    + \bigl(\log_tN+\Oh{1}\bigr) \bfM
    + 2\bigl(\log_tN+\Oh{1}\bigr) \bfA
    + \Oh{\log N} \bfS
    + \Oh{\log N} \bfO
  \end{equation}
  power series operations,
  \begin{equation*}
    N (\log N)^2 \, 2^{\Oh{\log^* N}}
  \end{equation*}
  operations in the ring of integers,
  and with
  \begin{equation}\label{eq:cost-bit-operations}
    N^2 (\log N)^4 \, 2^{\Oh{\log^* N}}
  \end{equation}
  bit operations.
\end{theorem}

In order to prove Theorem~\ref{thm:extract-coeff}---the complete proof can
be found in Section~\ref{sec:proofs},---we look at
the cost of calculating the first $N$ terms, which is done by extracting
coefficients of the power series
\begin{equation}\label{eq:power-series}
  H(q) = \sum_{n=0}^\infty g_t(n) q^n =
    \frac{\sum_{j=0}^\infty q^{\hp{j}} (-1)^j \prod_{i=1}^j
    \frac{q^{\hp{i}}}{1-q^{\hp{i}}}}{
    \sum_{j=0}^\infty (-1)^j \prod_{i=1}^j
    \frac{q^{\hp{i}}}{1-q^{\hp{i}}}}
\end{equation}
with
\begin{equation}\label{eq:hp}
  \hp{j} \colonequals 1+t+\dots+t^{j-1}.
\end{equation}
This generating function~\eqref{eq:power-series} can be found in
Flajolet and Prodinger~\cite[Theorem~2]{Flajolet-Prodinger:1987:level} for $t=2$ and in
Elsholtz, Heuberger and
Prodinger~\cite[Theorem~6]{Elsholtz-Heuberger-Prodinger:2013:huffm} for general~$t$.
It is derived from the equivalent formulation as
counting problem on trees, which was mentioned in the introduction.

\section{Auxiliary results}

When extracting the first $N$ coefficients, we do not need the ``full''
generating function, i.e., the infinite sums in the numerator and denominator
of~\eqref{eq:power-series}
can be truncated to finite sums. The following lemma tells us how many
coefficients we need. We use this asymptotic result in our analysis of
the algorithm; for the actual computer programme, we can check indices
and exponents by a direct computation.

\begin{lemma}\label{lem:bound-index-summands}
  To calculate numerator and denominator of the generating
  function~\eqref{eq:power-series} with precision~$N$, we need only summands
  with
  \begin{equation*}
    j \leq J = \log_t N + \Oh{1}.
  \end{equation*}
\end{lemma}

\begin{proof}
  Because of an additional factor $q^{\hp{j}}$ in each summand of the
  numerator, it is sufficient that the largest index of the denominator
  is less than~$N$. Therefore, we will only look at the indices of the
  denominator.

  Consider the summand of the denominator with index~$j$. The lowest index of a
  non-zero coefficient of the denominator is
  \begin{equation*}
    \sigma_j = \sum_{i=1}^j \hp{i} 
    = \frac{t^{j+1}-t}{(t-1)^2} - \frac{j}{t-1}
    = \frac{t^{j+1}}{(t-1)^2}
    \Bigl(1 - \frac{j(t-1)}{t^{j+1}} - \frac{1}{t^j}\Bigr)
  \end{equation*}
  where the notation $\hp{i}$ is defined in Equation~\eqref{eq:hp}. We only
  need summands with $\sigma_j < N$. Taking the logarithm yields
  \begin{equation*}
    j - 1
    + \log_t\Bigl(1 - \frac{j(t-1)}{t^{j+1}} - \frac{1}{t^j}\Bigr)
    - 2\log_t \Bigl(1-\frac{1}{t}\Bigr)
    < \log_t N.
  \end{equation*}
  As the first logarithm tends to $0$ as $j\to\infty$ and the second
  is bounded, the error term~$\Oh{1}$ is large enough and the result
  follows.
\end{proof}

While the bit size of the coefficients~$\f{g_t}{n}$ is linear in~$N$, the
size of the coefficients of numerator and denominator
of~\eqref{eq:power-series} is much smaller. We make this precise by using
the following lemma.

\begin{lemma}\label{lemma:coefficients-product-geometric}
  For $n \le N$, the $n$th coefficient of
  \begin{equation}\label{eq:product-geometric}
    \frac{1}{1-q^{\hp{1}}}
    \frac{1}{1-q^{\hp{2}}}
    \frac{1}{1-q^{\hp{3}}}
    \cdots
    \frac{1}{1-q^{\hp{j}}}
  \end{equation}
  with $j \le J$ and $J$ of Lemma~\ref{lem:bound-index-summands}
  as well as the $n$th coefficients of numerator and denominator
  of~\eqref{eq:power-series} can be written with
  \begin{equation}\label{eq:coefficients-bound-bits}
    \frac{(\log N)^2}{2(\log 2)(\log t)} + \Oh{\log N}
  \end{equation}
  bits.
\end{lemma}

\begin{proof}
  We start proving the claimed result for~\eqref{eq:product-geometric}
  and postpone handling numerator and denominator
  of~\eqref{eq:power-series} to the end of this proof.

  Each factor of~\eqref{eq:product-geometric} is a geometric series
  whose coefficients are either $0$ or $1$ and whose constant coefficient is $1$.
  In particular, these coefficients are nonnegative.
  Therefore,
  it suffices to show the result for $j=J$.

  As the coefficients are either $0$ or $1$, the $n$th
  coefficient of the product equals the cardinality of the set
  \begin{equation*}
    \setm[\bigg]{(a_1,a_2,\dots,a_J) \in \N_0^J}{\sum_{i=1}^J a_i \hp{i} = n}.
  \end{equation*}
  By using the crude estimate~$a_i \leq n/\hp{i}$, we see that we have at most
  $2N/\hp{i}$ choices for $a_i$ because $n<N$ and $\hp{i}<N$ by construction. Thus
  we can bound the
  cardinality of the set above by
  \begin{equation*}
    \frac{2^JN^J}{\hp{1} \hp{2} \cdots \hp{J}}
    \leq \frac{2^JN^J}{1 \cdot t \cdot t^2 \cdots t^{J-1}}
    = \frac{2^JN^J}{t^{(J-1)J/2}}.
  \end{equation*}
  We use $J=\log_tN + \Oh{1}$ of
  Lemma~\ref{lem:bound-index-summands} to obtain
  \begin{align*}
    \frac{2^JN^J}{t^{(J-1)J/2}}
    &= \fexp[\Big]{J (\log N) - J^2 \frac{\log t}{2} + J \frac{\log t}{2}
      + J (\log 2)} \\
    &\leq \fexp[\Big]{\frac{(\log N)^2}{\log t}
      - \frac{(\log N)^2}{2\log t} + \Oh{\log N}}
  \end{align*}
  from which follows that the $n$th coefficient
  of~\eqref{eq:product-geometric} is bounded by
  \begin{equation}\label{eq:coefficients-bound}
    \fexp[\Big]{\frac{(\log N)^2}{2\log t} + \Oh{\log N}}.
  \end{equation}
  The result in terms of bit size follows by taking the logarithm.

  Numerator and denominator are sums where
  $J$ summands are added up (or subtracted). This corresponds
  to an additional factor~$J$ in the bound~\eqref{eq:coefficients-bound}
  or an additional summand $\log_2J$ in the
  formula~\eqref{eq:coefficients-bound-bits}, respectively. As $J=\Oh{\log_tN}$
  by Lemma~\ref{lem:bound-index-summands}, this
  is absorbed by the error term, so the same formula holds.
\end{proof}

\section{Proof of Theorem~\ref{thm:extract-coeff}}
\label{sec:proofs}

  We start with an overview of our strategy.
  For computing the first~$N$ coefficients of the generating function~$H(q)$
  (see~\eqref{eq:power-series}), we only need the summands of
  the numerator and the denominator with $j<J$ according to
  Lemma~\ref{lem:bound-index-summands}.

  First, consider the denominator of $H(q)$. We
  compute the products
  \begin{equation*}
    \prod_{i=1}^j \frac{q^{\hp{i}}}{1-q^{\hp{i}}},
  \end{equation*}
  iteratively by expanding the $J$ different terms
  $q^{\hp{i}} / (1-q^{\hp{i}})$ as geometric series and perform power
  series multiplications. After each multiplication, we accumulate the
  result by using one power series addition.

  We deal with the numerator in the same fashion. However, by
  performing the computation of numerator and denominator
  simultaneously, the above products only need to be evaluated once.

  Finally, to obtain the first $N$ coefficients of $H(q)$, we need one power
  series division of numerator and denominator.

  Pseudocode for our algorithm is given in Algorithm~\ref{algorithm:pseudo-code};
  an efficient implementation using the FLINT library is
  presented in Appendix~\ref{sec:code}. The actual analysis of this
  algorithm is done by counting the operations needed, in particular
  the power series operations, and providing bounds for the bit sizes of
  the variables.

\begin{algorithm}
  \input{pseudocode.tex}
  \caption{Pseudocode}\label{algorithm:pseudo-code}
\end{algorithm}

Let us come to the actual proof.

\begin{proof}[Proof of Theorem~\ref{thm:extract-coeff}]

  We analyse the code of Algorithm~\ref{algorithm:pseudo-code}; see Appendix~\ref{sec:code}
  for the details. It starts by
  initialising variables (memory allocation and initial values) for
  the power series operations, which contributes $\Oh{1}\bfS$. Further
  initialisation is done by $\Oh{1}\bfO$ operations.

  For computing the first~$N$ coefficients of $H(q)$
  (see~\eqref{eq:power-series}), we only need the summands of
  numerator and denominator with $j<J=\log_tN+\Oh{1}$ according to
  Lemma~\ref{lem:bound-index-summands}. Speaking in terms of our
  computer programme, our outer loop needs $J$~passes. We now describe
  what happens in each of these passes; the final cost needs then to be
  multiplied by~$J$.

  Suppose we are in step~$j$.
  After some update of auxiliary variables (needing $\Oh{1}\bfO$~operations), we
  compute the product
  \begin{equation*}
    \prod_{i=1}^j \frac{q^{\hp{i}}}{1-q^{\hp{i}}},
  \end{equation*}
  out of the product with factors up to index~$i=j-1$. Expanding
  $q^{\hp{j}} / (1-q^{\hp{j}})$ as geometric series contributes at
  most~$\bfS$ and performing a power series multiplication
  contributes~$\bfM$ and additionally one swap~$\Oh{1}\bfS$. For
  obtaining the number of bit operations, we need estimates of the
  coefficients appearing in the multiplication.
  Lemma~\ref{lemma:coefficients-product-geometric} bounds their value by
  \begin{equation}\label{eq:bits-multiplication}
    \frac{(\log N)^2}{2(\log 2)(\log t)} + \Oh{\log N}
  \end{equation}
  bits. Therefore each of our power series multiplications~$\bfM$ needs
  \begin{equation*}
    N \log N\, 2^{\Oh{\log^* N}}
    \cdot (\log N)^2 \log\log N
    = N (\log N)^3 \log\log N \, 2^{\Oh{\log^* N}}
  \end{equation*}
  bit operations by the results of
  Fürer~\cite{Fuerer:2007:faster-integer-mult,
    Fuerer:2009:faster-integer-mult} and
  Harvey and van der Hoeven~\cite{Harvey-VanDerHoeven:2019:integer-mult-nlogn}; see also
  Section~\ref{sec:cost-add-mult}.
  
  After each multiplication, we accumulate the results for numerator
  and denominator by using one power series addition~$\bfA$ for each
  of the two. For the numerator, we additionally need $\Oh{1}\bfS$ operations for the
  multiplication by $q^{\hp{j}}$ performed by shifting.
  Concerning bit operations, we use the bound of the coefficients
  for numerator and denominator provided by
  Lemma~\ref{lemma:coefficients-product-geometric}.
  In terms of bit size, this leads to
  the number of bits given in~\eqref{eq:bits-multiplication}.
  Therefore a power series addition needs
  \begin{equation*}
    \Oh[\big]{N} \Oh[\big]{(\log N)^2} = \Oh[\big]{N (\log N)^2}
  \end{equation*}
  bit operations.

  In total, we end up with
  \begin{equation*}
    J\bigl(\bfM + 2\bfA + \Oh{1}\bfS + \Oh{1}\bfO\bigr)
  \end{equation*}
  operations to evaluate the outer loop; these operations translate to
  \begin{equation*}
    N (\log N)^2 \, 2^{\Oh{\log^* N}}    
  \end{equation*}
  operations in the ring of integers and to
  \begin{equation*}
     N (\log N)^4 \log\log N \, 2^{\Oh{\log^* N}}
  \end{equation*}
  bit operations.

  We are now ready to collect all costs for proving the first part of
  Theorem~\ref{thm:extract-coeff}.
  Additionally to the above,
  we divide the numerator by the denominator and
  need one power series division~$\bfD$. The clean-up accounts to
  $\Oh{1}\bfS$. This yields~\eqref{eq:cost-ps-operations}.

  Using the Newton--Raphson-method and Fürer's algorithm (see
  Section~\ref{sec:cost-add-mult} and Table~\ref{tab:cost-ops})
  a power series division~$\bfD$ results in
  \begin{equation*}
    N \log N \, 2^{\Oh{\log^* N}}
  \end{equation*}
  operations in the ring. Its operands\footnote{The actual bit size
    during the division is
    $\frac{N(\log N)^2}{2(\log 2 )(\log t)} + \Oh{N\log N}$; see the end
    of Section~\ref{sec:cost-add-mult} for details.} have bit size
  \begin{equation*}
    \frac{(\log N)^2}{2(\log 2 )(\log t)} + \Oh{\log N}
  \end{equation*}
  which results in
  \begin{equation*}
    N^2(\log N)^4 2^{\Oh{\log^* N}}
  \end{equation*}
  bit operations for our computations.
  
  We note that the number of bit operations of a power series
  operation~$\bfS$ is linear in $N$ as the coefficients are bounded
  and that $\bfO$ is an operation on numbers with $\Oh{\log N}$ bits.
  The error term includes all these. Collecting all
  bit operation results gives the upper bound~\eqref{eq:cost-bit-operations}.
\end{proof}

\section{Remarks}
\label{sec:remarks}

In this last section, we provide some remarks related to the
above proof and coefficient extraction algorithm.

\begin{remark}
In the proof above, we have seen that the cost (bit operations)
of the power series
division is asymptotically roughly (not taking logarithms and smaller
factors into account) a factor~$N$ larger than the cost for computing
numerator and denominator, and all the overhead cost.

Moreover, only focusing on the computation of numerator and
denominator, the costs (again bit operations) for computing these two
are asymptotically dominated by power series
multiplication, albeit only by roughly (again not taking into account
logarithmically smaller factors) a factor~$\log N$ compared to addition and
other power series operations.

Note that when only considering operations in the integer ring, then
the multiplications performed in the evaluation of numerator and
denominator take the asymptotically leading role by
$N (\log N)^2 \, 2^{\Oh{\log^* N}}$ operations compared to
$N (\log N) \, 2^{\Oh{\log^* N}}$ ring operations of the power series
division.
\end{remark}

At the end of this article we make a short remark on the memory requirements
for the presented coefficient extraction algorithm.

\begin{remark}\label{rem:space}
  Our algorithm needs $\Oh{N}$ units of memory---a unit stands for the
  memory requirements of storing a number bounded by~$\rho^N$---plus the
  memory needed for the power series multiplication and
  division.\footnote{We have been unable to find a reference for the memory
    requirements of, for example, the Schönhage--Strassen-algorithm. It seems
    that the \href{http://gmplib.org/}{\emph{GNU Multiple Precision Arithmetic
        Library (GMP)}} can do this with $12N$ units of memory;
    see~\cite{Granlund:gmp:memory-ssa} for a comment of one of its authors.}
  The above means that we can bound the memory requirements by $\Oh{N^2}$ bits.
\end{remark}

\renewcommand{\MR}[1]{}
\bibliographystyle{bibstyle/amsplainurl}
\bibliography{bib/cheub}

\appendix
\clearpage

\section{Code}
\label{sec:code}

Below are the relevant lines of a programme written in C for computing
the coefficients~$\f{g_t}{n}$ with $n<N$. The code can
be found at
\url{https://gitlab.com/dakrenn/count-nonequivalent-compact-huffman-codes}.
The programme
uses FLINT~\cite{flint:2015:2.5.2, Hart:2010:FLINT-intro}. Note that we do not
use aliasing of input and output arguments in multiplication because providing
our own auxiliary polynomial brings tiny performance improvements.

\begin{algorithm}
  {\tiny\input{code-excerpt.tex}}
  \caption{C-code using FLINT}\label{algorithm:code-excerpt}
\end{algorithm}

\section{Timing}
\label{sec:timing}

The table below contains timings (in seconds) for computing the first
$N$ coefficients with $t=2$.
\begin{center}
  \vspace*{1ex}
  \begin{tabular}{rr|rrr}
    $t$ & $N$ & $t_{\mathrm{n\&d}}$ & $t_{\mathrm{division}}$ & $t_{\mathrm{total}}$ \\
    \hline
    \input{c/timings-c_t2-table}
  \end{tabular}
  \vspace*{1ex}
\end{center}
Here, $t_{\mathrm{n\&d}}$ is the time for generating numerator and
denominator, $t_{\mathrm{division}}$ for the one power series division
and $t_{\mathrm{total}} = t_{\mathrm{n\&d}} + t_{\mathrm{division}}$.

The benchmark was executed on an Intel(R) Xeon(R) CPU E5-2630 v3 at
2.40GHz. The limiting factor for our computations is the memory
requirement; it is the reason computing at most~$N=2^{17}=131072$
coefficients.

The timings in the table and the theoretical result of this article
fit together; we can see the $\Oh{N^{2+\eps}}$ running time of the
algorithm in our implementation.
\vspace*{2ex}
\end{document}

%% file: pseudocode.tex
\begin{Verbatim}[commandchars=\\\{\}]
    \PY{n}{j} \PY{o}{=} \PY{l+m+mi}{0}\PY{p}{;} \PY{n}{hp\PYZus{}j} \PY{o}{=} \PY{l+m+mi}{0}\PY{p}{;} \PY{n}{sigma\PYZus{}j} \PY{o}{=} \PY{l+m+mi}{0}\PY{p}{;} \PY{n}{sign} \PY{o}{=} \PY{l+m+mi}{1}

    \PY{n}{numerator} \PY{o}{=} \PY{l+m+mi}{1}\PY{p}{;} \PY{n}{denominator} \PY{o}{=} \PY{l+m+mi}{1}\PY{p}{;} \PY{n}{coefficient\PYZus{}product} \PY{o}{=} \PY{l+m+mi}{1}

    \PY{k}{while} \PY{n+nb+bp}{True}\PY{p}{:}
        \PY{n}{j} \PY{o}{+}\PY{o}{=} \PY{l+m+mi}{1}\PY{p}{;} \PY{n}{hp\PYZus{}j} \PY{o}{=} \PY{l+m+mi}{1} \PY{o}{+} \PY{n}{t}\PY{o}{*}\PY{n}{hp\PYZus{}j}\PY{p}{;} \PY{n}{sigma\PYZus{}j} \PY{o}{+}\PY{o}{=} \PY{n}{hp\PYZus{}j}\PY{p}{;} \PY{n}{sign} \PY{o}{=} \PY{o}{\PYZhy{}}\PY{n}{sign}
        \PY{c+c1}{\PYZsh{} hp\PYZus{}j = [j]; sigma\PYZus{}j = sum\PYZus{}\PYZob{}i=1\PYZcb{}\PYZca{}j [j]; sign = (\PYZhy{}1)\PYZca{}j}

        \PY{k}{if} \PY{n}{sigma\PYZus{}j} \PY{o}{\PYZgt{}} \PY{n}{N}\PY{p}{:}
            \PY{k}{break}

        \PY{n}{new\PYZus{}factor} \PY{o}{=} \PY{l+m+mi}{0}
        \PY{k}{for} \PY{n}{i} \PY{o+ow}{in} \PY{n+nb}{range}\PY{p}{(}\PY{n}{hp\PYZus{}j}\PY{p}{,} \PY{n}{hp\PYZus{}j}\PY{p}{,} \PY{n}{N}\PY{p}{)}\PY{p}{:}
            \PY{n}{new\PYZus{}factor}\PY{p}{[}\PY{n}{i}\PY{p}{]} \PY{o}{=} \PY{l+m+mi}{1}
        \PY{c+c1}{\PYZsh{} new\PYZus{}factor = q\PYZca{}[j]/(1\PYZhy{}q\PYZca{}[j])}

        \PY{n}{coefficient\PYZus{}product} \PY{o}{*}\PY{o}{=} \PY{n}{new\PYZus{}factor}
        \PY{c+c1}{\PYZsh{} coefficient\PYZus{}product = prod\PYZus{}\PYZob{}i=1\PYZcb{}\PYZca{}j q\PYZca{}[i]/(1\PYZhy{}q\PYZca{}[i])}

        \PY{n}{denominator} \PY{o}{+}\PY{o}{=} \PY{n}{sign} \PY{o}{*} \PY{n}{coefficient\PYZus{}product}
        \PY{n}{numerator} \PY{o}{+}\PY{o}{=} \PY{n}{sign} \PY{o}{*} \PY{n}{coefficient\PYZus{}product}\PY{o}{.}\PY{n}{shifted}\PY{p}{(}\PY{n}{hp\PYZus{}j}\PY{p}{)}
        \PY{c+c1}{\PYZsh{} numerator up to summand j; denominator up to summand j}

    \PY{n}{result} \PY{o}{=} \PY{n}{numerator} \PY{o}{/} \PY{n}{denominator}
\end{Verbatim}

%% file: code-excerpt.tex
\begin{Verbatim}[commandchars=\\\{\}]
    \PY{n}{fmpz\PYZus{}poly\PYZus{}init}\PY{p}{(}\PY{n}{result}\PY{p}{)}\PY{p}{;}
    \PY{n}{fmpz\PYZus{}poly\PYZus{}init2}\PY{p}{(}\PY{n}{coefficient\PYZus{}product}\PY{p}{,} \PY{n}{N}\PY{p}{)}\PY{p}{;}
    \PY{n}{fmpz\PYZus{}poly\PYZus{}init2}\PY{p}{(}\PY{n}{coefficient\PYZus{}product\PYZus{}shifted}\PY{p}{,} \PY{l+m+mi}{2}\PY{o}{*}\PY{n}{N}\PY{p}{)}\PY{p}{;}
    \PY{n}{fmpz\PYZus{}poly\PYZus{}init2}\PY{p}{(}\PY{n}{new\PYZus{}factor}\PY{p}{,} \PY{n}{N}\PY{p}{)}\PY{p}{;}
    \PY{n}{fmpz\PYZus{}poly\PYZus{}init2}\PY{p}{(}\PY{n}{numerator}\PY{p}{,} \PY{n}{N}\PY{p}{)}\PY{p}{;}
    \PY{n}{fmpz\PYZus{}poly\PYZus{}init2}\PY{p}{(}\PY{n}{denominator}\PY{p}{,} \PY{n}{N}\PY{p}{)}\PY{p}{;}
    \PY{n}{fmpz\PYZus{}poly\PYZus{}init2}\PY{p}{(}\PY{n}{aux}\PY{p}{,} \PY{n}{N}\PY{p}{)}\PY{p}{;}

    \PY{n}{fmpz\PYZus{}poly\PYZus{}set\PYZus{}ui}\PY{p}{(}\PY{n}{numerator}\PY{p}{,} \PY{l+m+mi}{1}\PY{p}{)}\PY{p}{;}
    \PY{n}{fmpz\PYZus{}poly\PYZus{}set\PYZus{}ui}\PY{p}{(}\PY{n}{denominator}\PY{p}{,} \PY{l+m+mi}{1}\PY{p}{)}\PY{p}{;}
    \PY{n}{fmpz\PYZus{}poly\PYZus{}set\PYZus{}ui}\PY{p}{(}\PY{n}{coefficient\PYZus{}product}\PY{p}{,} \PY{l+m+mi}{1}\PY{p}{)}\PY{p}{;}

    \PY{n}{j}\PY{o}{=}\PY{l+m+mi}{0}\PY{p}{;}
    \PY{n}{hp\PYZus{}j}\PY{o}{=}\PY{l+m+mi}{0}\PY{p}{;}
    \PY{n}{sigma\PYZus{}j}\PY{o}{=}\PY{l+m+mi}{0}\PY{p}{;}
    \PY{n}{sign}\PY{o}{=}\PY{l+m+mi}{1}\PY{p}{;}

    \PY{k}{while}\PY{p}{(}\PY{o}{\PYZhy{}}\PY{l+m+mi}{1}\PY{p}{)} \PY{p}{\PYZob{}}
        \PY{n}{j}\PY{o}{+}\PY{o}{+}\PY{p}{;}
        \PY{n}{hp\PYZus{}j}\PY{o}{=}\PY{l+m+mi}{1}\PY{o}{+}\PY{n}{t}\PY{o}{*}\PY{n}{hp\PYZus{}j}\PY{p}{;}
        \PY{n}{sigma\PYZus{}j}\PY{o}{+}\PY{o}{=}\PY{n}{hp\PYZus{}j}\PY{p}{;}
        \PY{n}{sign}\PY{o}{=}\PY{o}{\PYZhy{}}\PY{n}{sign}\PY{p}{;}
        \PY{c+cm}{/* hp\PYZus{}j=[j]; sigma\PYZus{}j=sum\PYZus{}\PYZob{}i=1\PYZcb{}\PYZca{}j [j]; sign=(\PYZhy{}1)\PYZca{}j */}

        \PY{k}{if}\PY{p}{(}\PY{n}{sigma\PYZus{}j}\PY{o}{\PYZgt{}}\PY{o}{=}\PY{n}{N}\PY{p}{)} \PY{p}{\PYZob{}}
            \PY{k}{break}\PY{p}{;}
        \PY{p}{\PYZcb{}}

        \PY{k}{for}\PY{p}{(}\PY{n}{i}\PY{o}{=}\PY{n}{hp\PYZus{}j}\PY{p}{;} \PY{n}{i}\PY{o}{\PYZlt{}}\PY{n}{N}\PY{p}{;} \PY{n}{i}\PY{o}{+}\PY{o}{=}\PY{n}{hp\PYZus{}j}\PY{p}{)} \PY{p}{\PYZob{}}
            \PY{n}{fmpz\PYZus{}poly\PYZus{}set\PYZus{}coeff\PYZus{}ui}\PY{p}{(}\PY{n}{new\PYZus{}factor}\PY{p}{,} \PY{n}{i}\PY{p}{,} \PY{l+m+mi}{1}\PY{p}{)}\PY{p}{;}
        \PY{p}{\PYZcb{}}
        \PY{c+cm}{/* new\PYZus{}factor = q\PYZca{}[j]/(1\PYZhy{}q\PYZca{}[j]) */}

        \PY{n}{fmpz\PYZus{}poly\PYZus{}mullow}\PY{p}{(}\PY{n}{aux}\PY{p}{,} \PY{n}{coefficient\PYZus{}product}\PY{p}{,} \PY{n}{new\PYZus{}factor}\PY{p}{,} \PY{n}{N}\PY{p}{)}\PY{p}{;}
        \PY{n}{fmpz\PYZus{}poly\PYZus{}swap}\PY{p}{(}\PY{n}{coefficient\PYZus{}product}\PY{p}{,} \PY{n}{aux}\PY{p}{)}\PY{p}{;}
        \PY{c+cm}{/* coefficient\PYZus{}product=prod\PYZus{}\PYZob{}i=1\PYZcb{}\PYZca{}j q\PYZca{}[i]/(1\PYZhy{}q\PYZca{}[i]) */}

        \PY{k}{for}\PY{p}{(}\PY{n}{i}\PY{o}{=}\PY{n}{hp\PYZus{}j}\PY{p}{;} \PY{n}{i}\PY{o}{\PYZlt{}}\PY{n}{N}\PY{p}{;} \PY{n}{i}\PY{o}{+}\PY{o}{=}\PY{n}{hp\PYZus{}j}\PY{p}{)} \PY{p}{\PYZob{}}
            \PY{n}{fmpz\PYZus{}poly\PYZus{}set\PYZus{}coeff\PYZus{}ui}\PY{p}{(}\PY{n}{new\PYZus{}factor}\PY{p}{,} \PY{n}{i}\PY{p}{,} \PY{l+m+mi}{0}\PY{p}{)}\PY{p}{;}
        \PY{p}{\PYZcb{}}
        \PY{c+cm}{/* new\PYZus{}factor = 0 */}

        \PY{k}{if}\PY{p}{(}\PY{n}{sign}\PY{o}{=}\PY{o}{=}\PY{l+m+mi}{1}\PY{p}{)} \PY{p}{\PYZob{}}
            \PY{n}{fmpz\PYZus{}poly\PYZus{}add}\PY{p}{(}\PY{n}{denominator}\PY{p}{,} \PY{n}{denominator}\PY{p}{,} \PY{n}{coefficient\PYZus{}product}\PY{p}{)}\PY{p}{;}

            \PY{k}{if}\PY{p}{(}\PY{n}{hp\PYZus{}j}\PY{o}{+}\PY{n}{sigma\PYZus{}j}\PY{o}{\PYZlt{}}\PY{n}{N}\PY{p}{)} \PY{p}{\PYZob{}}
                \PY{n}{fmpz\PYZus{}poly\PYZus{}set}\PY{p}{(}\PY{n}{coefficient\PYZus{}product\PYZus{}shifted}\PY{p}{,} \PY{n}{coefficient\PYZus{}product}\PY{p}{)}\PY{p}{;}
                \PY{n}{fmpz\PYZus{}poly\PYZus{}truncate}\PY{p}{(}\PY{n}{coefficient\PYZus{}product\PYZus{}shifted}\PY{p}{,} \PY{n}{N}\PY{o}{\PYZhy{}}\PY{n}{hp\PYZus{}j}\PY{p}{)}\PY{p}{;}
                \PY{n}{fmpz\PYZus{}poly\PYZus{}shift\PYZus{}left}\PY{p}{(}\PY{n}{coefficient\PYZus{}product\PYZus{}shifted}\PY{p}{,} \PY{n}{coefficient\PYZus{}product\PYZus{}shifted}\PY{p}{,} \PY{n}{hp\PYZus{}j}\PY{p}{)}\PY{p}{;}
                \PY{n}{fmpz\PYZus{}poly\PYZus{}add}\PY{p}{(}\PY{n}{numerator}\PY{p}{,} \PY{n}{numerator}\PY{p}{,} \PY{n}{coefficient\PYZus{}product\PYZus{}shifted}\PY{p}{)}\PY{p}{;}
            \PY{p}{\PYZcb{}}
        \PY{p}{\PYZcb{}} \PY{k}{else} \PY{p}{\PYZob{}}
            \PY{n}{fmpz\PYZus{}poly\PYZus{}sub}\PY{p}{(}\PY{n}{denominator}\PY{p}{,} \PY{n}{denominator}\PY{p}{,} \PY{n}{coefficient\PYZus{}product}\PY{p}{)}\PY{p}{;}

            \PY{k}{if}\PY{p}{(}\PY{n}{hp\PYZus{}j}\PY{o}{+}\PY{n}{sigma\PYZus{}j}\PY{o}{\PYZlt{}}\PY{n}{N}\PY{p}{)} \PY{p}{\PYZob{}}
                \PY{n}{fmpz\PYZus{}poly\PYZus{}set}\PY{p}{(}\PY{n}{coefficient\PYZus{}product\PYZus{}shifted}\PY{p}{,} \PY{n}{coefficient\PYZus{}product}\PY{p}{)}\PY{p}{;}
                \PY{n}{fmpz\PYZus{}poly\PYZus{}truncate}\PY{p}{(}\PY{n}{coefficient\PYZus{}product\PYZus{}shifted}\PY{p}{,} \PY{n}{N}\PY{o}{\PYZhy{}}\PY{n}{hp\PYZus{}j}\PY{p}{)}\PY{p}{;}
                \PY{n}{fmpz\PYZus{}poly\PYZus{}shift\PYZus{}left}\PY{p}{(}\PY{n}{coefficient\PYZus{}product\PYZus{}shifted}\PY{p}{,} \PY{n}{coefficient\PYZus{}product\PYZus{}shifted}\PY{p}{,} \PY{n}{hp\PYZus{}j}\PY{p}{)}\PY{p}{;}
                \PY{n}{fmpz\PYZus{}poly\PYZus{}sub}\PY{p}{(}\PY{n}{numerator}\PY{p}{,} \PY{n}{numerator}\PY{p}{,} \PY{n}{coefficient\PYZus{}product\PYZus{}shifted}\PY{p}{)}\PY{p}{;}
            \PY{p}{\PYZcb{}}
        \PY{p}{\PYZcb{}}
        \PY{c+cm}{/* numerator up to summand j; denominator up to summand j */}
    \PY{p}{\PYZcb{}}

    \PY{n}{fmpz\PYZus{}poly\PYZus{}clear}\PY{p}{(}\PY{n}{coefficient\PYZus{}product}\PY{p}{)}\PY{p}{;}
    \PY{n}{fmpz\PYZus{}poly\PYZus{}clear}\PY{p}{(}\PY{n}{coefficient\PYZus{}product\PYZus{}shifted}\PY{p}{)}\PY{p}{;}
    \PY{n}{fmpz\PYZus{}poly\PYZus{}clear}\PY{p}{(}\PY{n}{new\PYZus{}factor}\PY{p}{)}\PY{p}{;}
    \PY{n}{fmpz\PYZus{}poly\PYZus{}clear}\PY{p}{(}\PY{n}{aux}\PY{p}{)}\PY{p}{;}

    \PY{n}{fmpz\PYZus{}poly\PYZus{}div\PYZus{}series}\PY{p}{(}\PY{n}{result}\PY{p}{,} \PY{n}{numerator}\PY{p}{,} \PY{n}{denominator}\PY{p}{,} \PY{n}{N}\PY{p}{)}\PY{p}{;}

    \PY{n}{fmpz\PYZus{}poly\PYZus{}clear}\PY{p}{(}\PY{n}{numerator}\PY{p}{)}\PY{p}{;}
    \PY{n}{fmpz\PYZus{}poly\PYZus{}clear}\PY{p}{(}\PY{n}{denominator}\PY{p}{)}\PY{p}{;}
\end{Verbatim}

%% file: c/timings-c_t2-table.tex
2 & 256 & 0.000 & 0.001 & 0.001 \\
2 & 512 & 0.001 & 0.005 & 0.006 \\
2 & 1024 & 0.002 & 0.014 & 0.016 \\
2 & 2048 & 0.005 & 0.045 & 0.050 \\
2 & 4096 & 0.013 & 0.126 & 0.139 \\
2 & 8192 & 0.029 & 0.562 & 0.591 \\
2 & 16384 & 0.067 & 2.225 & 2.292 \\
2 & 32768 & 0.251 & 10.709 & 10.960 \\
2 & 65536 & 0.617 & 45.259 & 45.877 \\
2 & 131072 & 1.189 & 198.995 & 200.184 \\